\documentclass[a4paper,12pt,intlimits,oneside]{amsart}

\usepackage{enumerate}
\usepackage{amsfonts,amsmath}

\usepackage{latexsym,amssymb}

\usepackage{mathrsfs}

\newcommand{\comment}[1]{}

\newcommand{\bR}{{\mathbb R}}

\def\D{{\mathcal D}}
\def\G{{\mathcal G}}
\def\H{{\mathcal H}}

\def\a{\mathfrak a}
\def\b{\mathfrak b}

\def\n{\mathfrak n}

\def\M{{\mathcal M}}

\def\X{{\mathcal X}}

\begin{document}

\title[On the product of functions in $BMO$ and $H^1$]{On the product of functions in $BMO$ and $H^1$ over spaces of homogeneous type}         

\author{\bf Luong Dang Ky}  
\address{Department of Mathematics, University of Quy Nhon, 170 An Duong Vuong, Quy Nhon, Binh Dinh, Viet Nam} 
\email{{\tt dangky@math.cnrs.fr}}


\begin{abstract}
Let $\mathcal X$ be an RD-space, which means that $\mathcal X$ is a space of homogeneous type in the sense of Coifman-Weiss with the additional property that a reverse doubling property holds in $\mathcal X$. The aim of the present paper is to study the product of functions in $BMO$ and $H^1$ in this setting. Our results generalize some recent results in \cite{Feu} and \cite{LP}.

\end{abstract}

\maketitle
\newtheorem{theorem}{Theorem}[section]
\newtheorem{lemma}{Lemma}[section]
\newtheorem{proposition}{Proposition}[section]
\newtheorem{remark}{Remark}[section]
\newtheorem{corollary}{Corollary}[section]
\newtheorem{definition}{Definition}[section]
\newtheorem{example}{Example}[section]
\numberwithin{equation}{section}
\newtheorem{Theorem}{Theorem}[section]
\newtheorem{Lemma}{Lemma}[section]
\newtheorem{Proposition}{Proposition}[section]
\newtheorem{Remark}{Remark}[section]
\newtheorem{Corollary}{Corollary}[section]
\newtheorem{Definition}{Definition}[section]
\newtheorem{Example}{Example}[section]
\newtheorem*{theoremjj}{Theorem J-J}
\newtheorem*{theorema}{Theorem A}
\newtheorem*{theoremb}{Theorem B}
\newtheorem*{theoremc}{Theorem C}
\newtheorem*{conjecture}{Conjecture}
\newtheorem*{open question}{Open question}

\section{Introduction and statement of main results}

A famous result of C. Fefferman state that $BMO(\bR^n)$ is the dual space of $H^1(\bR^n)$. Although, for $f\in BMO(\bR^n)$ and $g\in H^1(\bR^n)$, the point-wise product $fg$ may not be an integrable function, one (see \cite{BIJZ}) can view the product of $f$ and $g$ as a distribution, denoted by $f\times g$. Such a distribution can be written as the sum of an integrable function and a distribution in a new Hardy space, so-called Hardy space of  Musielak-Orlicz type (see \cite{BGK, Ky1}). A complete study about the product of functions in $BMO$ and $\H^1$ has been firstly done by Bonami, Iwaniec, Jones and Zinsmeister \cite{BIJZ}. Recently, Li and Peng \cite{LP} generalized this study to the setting of Hardy and $BMO$ spaces associated with Schr\"odinger operators. In particular, Li and Peng showed that if $L= -\Delta + V$ is a Schr\"odinger operator with the potential $V$ belongs to the reverse H\"older class $RH_q$ for some $q\geq n/2$, then one can view the product of $b \in BMO_L(\bR^n)$ and $f\in H^1_L(\bR^n)$ as a distribution $b\times f$ which can be written the sum of an integrable function and a distribution in $H^{\wp}_L(\bR^n,d\mu)$. Here $H^{\wp}_L(\bR^n,d\mu)$ is the weighted Hardy-Orlicz space associated with $L$, related to the Orlicz function $\wp(t) =t/\log(e+t)$ and the weight $d\nu(x)= dx/\log(e+|x|)$. More precisely, they proved the following.

\begin{theorema}
For each $f\in H^1_L(\bR^n)$, there exist two bounded linear operators  ${\mathscr L}_f: BMO_L(\bR^n) \to L^1(\bR^n)$ and ${\mathscr H}_f: BMO_L(\bR^n) \to H^{\wp}_L(\bR^n,d\nu)$ such that for every $b\in BMO_L(\bR^n)$,
$$b\times f = {\mathscr L}_f(b) + {\mathscr H}_f(b).$$
\end{theorema}

Let $(\X,d,\mu)$ be a space of homogeneous type in the sense of Coifman-Weiss. Following Han, M\"uller and Yang \cite{HMY}, we say that $(\X,d,\mu)$ is an {\sl RD-space} if $\mu$ satisfies {\sl reverse doubling property}, i.e., there exists a constant $C>1$ such that for all $x\in \mathcal X$ and $r>0$,
$$\mu(B(x,2r))\geq C \mu(B(x,r)).$$
A typical example for such RD-spaces is the Carnot-Carath\'eodory space with doubling measure. We refer to the seminal paper of Han, M\"uller and Yang \cite{HMY} (see also \cite{GLY, GLY2, YYZ, YZ}) for a systematic study of the theory of function spaces in harmonic analysis on RD-spaces.

 Let $(\X,d,\mu)$ be an RD-space. Recently, in analogy with the classical result of Bonami-Iwaniec-Jones-Zinsmeister, Feuto proved in \cite{Feu} that:

\begin{theoremb}
For each $f\in H^1(\X)$, there exist two bounded linear operators  ${\mathscr L}_f: BMO(\X) \to L^1(\X)$ and ${\mathscr H}_f: BMO(\X) \to H^{\wp}(\X,d\nu)$ such that for every $b\in BMO(\X)$,
$$b\times f = {\mathscr L}_f(b) + {\mathscr H}_f(b).$$
\end{theoremb}

Here the weight $d\nu(x)=d\mu(x)/\log(e+ d(x_0,x))$ with $x_0\in \X$ and the Orlicz function $\wp$ is as in Theorem A. It should be pointed out that in \cite{Feu}, for $f=\sum_{j=1}^{\infty} \lambda_j a_j$,  the author  defined the distribution $b\times f$ as 
\begin{equation}\label{Feuto by Bonami}
b \times f:= \sum_{j=1}^\infty \lambda_j (b-b_{B_j})a_j + \sum_{j=1}^\infty \lambda_j b_{B_j} a_j
\end{equation}
by proving that the second series is convergent in $H^{\wp}(\X,d\nu)$. This is made possible by the fact that $H^{\wp}(\X,d\nu)$ is complete and is continuously imbedded into the space of distributions $(\G^{\epsilon}_0(\beta,\gamma))'$ (see Section 2), which is not established in \cite{Feu}. Moreover one has to prove that Definition (\ref{Feuto by Bonami}) does not depend on the atomic decomposition of $f$. In this paper, we give a definition for the distribution $b\times f$ (see Section 3) which is similar to that of Bonami-Iwaniec-Jones-Zinsmeister. 

Our first main result can be read as follows.

\begin{theorem}\label{the first main theorem}
For each $f\in H^1(\X)$, there exist two bounded linear operators  ${\mathscr L}_f: BMO(\X) \to L^1(\X)$ and ${\mathscr H}_f: BMO(\X) \to H^{\log}(\X)$ such that for every $b\in BMO(\X)$,
$$b\times f = {\mathscr L}_f(b) + {\mathscr H}_f(b).$$
\end{theorem}
 
Here $H^{\log}(\X)$ is the {\sl Musielak-Orlicz Hardy space} related to the Musielak-Orlicz function $\varphi(x,t)= \frac{t}{\log(e+ d(x_0,x)) + \log(e+t)}$ (see Section 2). Theorem \ref{the first main theorem} is an improvement of Theorem B since $H^{\log}(\X)$ is a proper subspace of $H^{\wp}(\X,d\nu)$.

Let $\rho$ be an {\sl admissible function} (see Section 2). Recently, Yang and Zhou \cite{YYZ, YZ} introduced and studied Hardy spaces and Morrey-Campanato spaces related to the function $\rho$. There, they established that $BMO_\rho(\X)$ is the dual space of $H^1_\rho(\X)$.  Similar to the classical case, we can define the product of functions $b\in BMO_\rho(\X)$ and $f\in H^1_\rho(\X)$ as distributions $b\times f \in (\G^{\epsilon}_0(\beta,\gamma))'$. 

Our next main result is as follows.

\begin{theorem}\label{the second main theorem}
For each $f\in H^1_\rho(\X)$, there exist two bounded linear operators  ${\mathscr L}_{\rho, f}: BMO_\rho(\X) \to L^1(\X)$ and ${\mathscr H}_{\rho, f}: BMO_\rho(\X) \to H^{\log}(\X)$ such that for every $b\in BMO_\rho(\X)$,
$$b\times f = {\mathscr L}_{\rho, f}(b) + {\mathscr H}_{\rho, f}(b).$$
\end{theorem}

When $\X\equiv\bR^n, n\geq 3,$ and $\rho(x)\equiv \sup\{r>0: \frac{1}{r^{n-2}}\int_{B(x,r)}V(y)dy\leq 1\}$, where $L= -\Delta +V$ is as in Theorem A, one has $BMO_\rho(\X)\equiv BMO_L(\bR^n)$ and $H^1_\rho(\X)\equiv H^1_L(\bR^n)$.
So, Theorem \ref{the second main theorem} is an improvement of Theorem A since $H^{\log}(\bR^n)$ is a proper subspace of $H^{\wp}_L(\bR^n, d\nu)$ (see \cite{Ky2}).

The following conjecture is suggested by A. Bonami and F. Bernicot.

\begin{conjecture}
There exist two bounded bilinear operators $\mathscr L: BMO(\X)\times H^1(\X) \to L^1(\X)$ and $\mathscr H: BMO(\X)\times H^1(\X)\to H^{\log}(\X)$ such that 
$$b\times f = \mathscr L(b,f) + \mathscr H(b,f).$$
\end{conjecture}

It should be pointed out that when $\X=\bR^n$ and $H^{\log}(\X)$ is replaced by $H^{\wp}(\bR^n, d\nu)$, the above conjecture is just Conjecture 1.7 of \cite{BIJZ}, which answered recently by Bonami, Grellier and Ky \cite{BGK} (see also \cite{Ky2}). 

Throughout the whole paper, $C$ denotes a positive geometric constant which is independent of the main parameters, but may change from line to line. We write $f\sim g$ if there exists a constant $C>1$ such that $C^{-1}f\leq g\leq C f$.

The paper is organized as follows. In Section 2, we present some notations and preliminaries about $BMO$ type spaces and Hardy type spaces on RD-spaces. Section 3 is devoted to prove Theorem \ref{the first main theorem}. Finally, we give the proof for Theorem \ref{the second main theorem} in Section 4.

{\bf Acknowledgements.} The author would like to thank  Aline Bonami, Sandrine Grellier, Dachun Yang and Fr\'ed\'eric Bernicot for very useful suggestions.

\section{Some preliminaries and notations}

 Let $d$ be a quasi-metric on a set $\X$, that is, $d$ is a nonnegative function on $\mathcal X\times \mathcal X$ satisfying
\begin{enumerate}[(a)]
\item $d(x,y)=d(y,x)$,
\item $d(x,y)>0$ if and only if $x\ne y$,
\item there exists a constant $\kappa\geq 1$ such that for all $x,y,z\in \mathcal X$,
\begin{equation}
d(x,z)\leq \kappa(d(x,y)+ d(y,z)).
\end{equation}
\end{enumerate}
A trip $(\mathcal X, d,\mu)$ is called a {\sl space of homogeneous type} in the sense of Coifman-Weiss \cite{CW} if $\mu$ is a regular Borel measure satisfying {\sl doubling property}, i.e. there exists a constant $C>1$ such that for all $x\in \mathcal X$ and $r>0$,
$$\mu(B(x,2r))\leq C \mu(B(x,r)).$$

Following Han, M\"uller and Yang \cite{HMY}, $(\mathcal X, d,\mu)$ is called an {\sl RD-space} if $(\mathcal X, d,\mu)$ is a space of homogeneous type and $\mu$ also satisfies {\sl reverse doubling property}, i.e. there exists a constant $C>1$ such that for all $x\in \mathcal X$ and $r>0$,
$$\mu(B(x,2r))\geq C \mu(B(x,r)).$$

Set $\mbox{diam}(\mathcal X) := \sup_{x,y\in\mathcal X} d(x,y)$. It should be pointed out that $(\mathcal X, d,\mu)$ is an RD-space if and only if  there exist constants $0<\mathfrak d \leq \mathfrak n$ and $C> 1$ such that for all $x\in\mathcal X$, $0<r<\mbox{diam}(\mathcal X)/2$, and $1\leq \lambda < \mbox{diam}(\mathcal X)/(2r)$,
\begin{equation}\label{RD-spaces}
C^{-1} \lambda^{\mathfrak d} \mu(B(x,r)) \leq \mu(B(x,\lambda r))\leq C \lambda^{\mathfrak n} \mu(B(x,r)).
\end{equation}

 Here and what in follows, for $x, y\in\X$ and $r>0$, we denote $V_r(x):= \mu(B(x,r))$ and $V(x,y):= \mu(B(x,d(x,y)))$.

\begin{definition}\label{definition for test functions}
Let $x_0\in\mathcal X$, $r>0$, $0<\beta\leq 1$ and $\gamma >0$. A function $f$ is said to belong to the space of test functions, $\mathcal G(x_0,r,\beta,\gamma)$, if there exists a positive constant $C_f$ such that
\begin{enumerate}[(i)]
\item $|f(x)| \leq C_f \frac{1}{V_r(x_0) + V(x_0,x)}\Big(\frac{r}{r+ d(x_0,x)}\Big)^\gamma$ for all $x\in\mathcal X$;

\item $|f(x) - f(y)|\leq C_f  \Big(\frac{d(x,y)}{r+ d(x_0,x)}\Big)^\beta \frac{1}{V_r(x_0) + V(x_0,x)}\Big(\frac{r}{r+ d(x_0,x)}\Big)^\gamma$ for all $x,y\in \mathcal X$ satisfying that $d(x,y)\leq \frac{r + d(x_0,x)}{2\kappa}$.
\end{enumerate}
For any $f\in \mathcal G(x_0,r,\beta,\gamma)$, we define 
$$\|f\|_{\mathcal G(x_0,r,\beta,\gamma)}:= \inf \{C_f: (i) \; \mbox{and} \;(ii) \;\mbox{hold}\}.$$
\end{definition}

Let $\rho$ be a positive function on $\X$. Following Yang and Zhou \cite{YZ}, the function $\rho$ is said to {\sl be admissible} if there exist positive constants $C_0$ and $k_0$ such that for all $x,y\in \X$,
$$\rho(y)\leq C_0 [\rho(x)]^{1/(1+k_0)} [\rho(x)+d(x,y)]^{k_0/(1+k_0)}.$$

{\sl Throughout the whole paper}, we always assume that $\mathcal X$ is an RD-space with $\mu(\mathcal X)=\infty$, and $\rho$ is an admissible function on $\X$. Also we fix $x_0\in \X$.

In Definition \ref{definition for test functions}, it is easy to see that $\mathcal G(x_0,1,\beta,\gamma)$ is a Banach space.  For simplicity, we write $\mathcal G(\beta,\gamma)$ instead of $\mathcal G(x_0,1,\beta,\gamma)$. Let $\epsilon\in (0,1]$ and $\beta,\gamma\in (0,\epsilon]$, we define the space $\mathcal G^\epsilon_0(\beta,\gamma)$ to be the completion of $\mathcal G(\epsilon,\epsilon)$ in $\mathcal G(\beta,\gamma)$, and denote by $(\mathcal G^\epsilon_0(\beta,\gamma))'$ the space of all continuous linear functionals on $\mathcal G^\epsilon_0(\beta,\gamma)$. We say that $f$ is a {\sl distribution} if $f$ belongs to $(\mathcal G^\epsilon_0(\beta,\gamma))'$.

Remark that, for any  $x\in \mathcal X$ and $r>0$, one has $\mathcal G(x,r,\beta,\gamma)= \mathcal G(x_0,1,\beta,\gamma)$  with equivalent norms, but of course the constants are depending on $x$ and $r$.

Let $f$ be a distribution in $(\mathcal G^\epsilon_0(\beta,\gamma))'$. We define {\sl the  grand maximal functions} $\M(f)$ and  $\M_\rho(f)$ as following
$$\M(f)(x) := \sup\{|\langle f,\varphi \rangle|: \varphi\in \mathcal G^\epsilon_0(\beta,\gamma), \|\varphi\|_{\mathcal G(x,r,\beta,\gamma)}\leq 1\; \mbox{for some}\; r>0\},$$
$$\M_\rho(f)(x) := \sup\{|\langle f,\varphi \rangle|: \varphi\in \mathcal G^\epsilon_0(\beta,\gamma), \|\varphi\|_{\mathcal G(x,r,\beta,\gamma)}\leq 1\; \mbox{for some}\; r\in (0,\rho(x))\}.$$

Let $L^{\log}(\X)$ (see \cite{BGK, Ky1} for details) be the Musielak-Orlicz type space of $\mu$-measurable functions $f$ such that
$$\int_{\X} \frac{|f(x)|}{\log(e+|f(x)|) +\log(e + d(x_0,x))} d\mu(x)<\infty.$$
For $f\in L^{\log}(\X)$, we define the "norm" of $f$ as
$$\|f\|_{L^{\log}}=\inf\left\{ \lambda>0:  \int_{\X} \frac{\frac{|f(x)|}{\lambda}}{\log(e+\frac{|f(x)|}{\lambda}) +\log(e + d(x_0,x))} d\mu(x)\leq 1\right\}.$$

\begin{definition}
Let $\epsilon\in (0,1)$ and $\beta,\gamma\in (0,\epsilon)$.
\begin{enumerate}[(i)]
\item The Hardy space $H^1(\mathcal X)$ is defined by
$$H^1(\mathcal X) = \{f\in (\mathcal G^\epsilon_0(\beta,\gamma))': \|f\|_{H^1}:= \|\M( f)\|_{L^1}<\infty \}.$$
\item The Hardy space $H^1_\rho(\mathcal X)$ is defined by
$$H^1_\rho(\mathcal X) = \{f\in (\mathcal G^\epsilon_0(\beta,\gamma))': \|f\|_{H^1_\rho}:= \|\M_\rho( f)\|_{L^1}<\infty \}.$$
\item The Hardy space $H^{\log}(\mathcal X)$ is defined by
$$H^{\log}(\mathcal X) = \{f\in (\mathcal G^\epsilon_0(\beta,\gamma))': \|f\|_{H^{\log}}:= \|\M( f)\|_{L^{\log}}<\infty \}.$$
\end{enumerate}
\end{definition}

It is clear that $H^1(\X) \subset H^1_{\rho}(\X)$ and $H^1(\X)\subset H^{\log}(\X)$ with the inclusions are continuous. It should be pointed out that the Musielak-Orlicz Hardy space $H^{\log}(\X)$ is a proper subspace of the weighted Hardy-Orlicz space $\H^{\wp}(\X, \nu)$ studied in \cite{Feu}. We refer to \cite{Ky1} for an introduction to Musielak-Orlicz Hardy spaces on the Euclidean space $\bR^n$.

\begin{definition}
Let $q\in (1,\infty]$.
\begin{enumerate}[(i)]
\item A measurable function $ \mathfrak a$ is called an $(H^1,q)$-atom related to the ball $B(x,r)$ if 
\begin{enumerate}[(a)]
\item supp $\a\subset B(x,r)$,
\item $\|\a\|_{L^q}\leq (V_r(x))^{1/q-1}$,
\item $\int_{\mathcal X} \a(y) d\mu(y)=0$.
\end{enumerate}

\item A measurable function $\mathfrak a$ is called an $(H^1_\rho,q)$-atom related to the ball $B(x,r)$ if $r \leq 2\rho(x)$ and $\mathfrak a$ satisfies (a) and (b), and when $r < \rho(x)$, $\mathfrak a$ also satisfies (c).

\end{enumerate}

\end{definition}

The following results were established in \cite{GLY, YZ}.

\begin{Theorem}\label{atomic decomposition}
Let  $\epsilon\in (0,1)$, $\beta,\gamma\in (0,\epsilon)$ and $q\in(1,\infty]$. Then, we have:
\begin{enumerate}[(i)]
\item The space $H^1(\mathcal X)$ coincides with the Hardy space $H^{1,q}_{\rm at}(\mathcal X)$ of Coifman-Weiss. More precisely, $f\in H^1(\X)$ if and only if $f$ can be written as  $f= \sum_{j=1}^\infty \lambda_j a_j$ where the $a_j$'s are $(H^1,q)$-atoms and $\{\lambda_j\}_{j=1}^\infty\in\ell^1$. Moreover,
$$\|f\|_{H^1}\sim \inf \left\{\sum_{j=1}^\infty |\lambda_j| : f= \sum_{j=1}^\infty \lambda_j a_j\right\}.$$
\item $f\in H^1_\rho(\X)$ if and only if $f$ can be written as  $f= \sum_{j=1}^\infty \lambda_j a_j$ where the $a_j$'s are $(H^1_\rho,q)$-atoms and $\{\lambda_j\}_{j=1}^\infty\in\ell^1$. Moreover,
$$\|f\|_{H^1_\rho}\sim \inf \left\{\sum_{j=1}^\infty |\lambda_j| : f= \sum_{j=1}^\infty \lambda_j a_j\right\}.$$

\end{enumerate}

\end{Theorem}

Here and what in follows, for any ball $B\subset \X$ and $g\in L^1_{\rm loc}(\X)$, we denote by $g_B$ the average value of $g$ over the ball $B$ and denote
$$MO(g,B):= \frac{1}{\mu(B)}\int_{B}|g(x) - g_B| d\mu(x).$$
Recall (see \cite{CW}) that a function $f\in L^1_{\rm loc}(\X)$ is said to be in $BMO(\X)$ if
$$\|f\|_{BMO}=\sup_{B} MO(f,B)<\infty,$$
where the supremum is taken all over balls $B\subset\X$.

\begin{definition}
Let $\rho$ be an admissible function and $\D:= \{B(x,r)\subset \X: r\geq \rho(x)\}$. A function $f\in L^1_{\rm loc}(\X)$ is said to be in $BMO_\rho(\X)$ if
$$\|f\|_{BMO_\rho}= \|f\|_{BMO} + \sup_{B\in \D}\frac{1}{\mu(B)}\int_B |f(x)| d\mu(x) <\infty.$$
\end{definition}

The following results are well-known, see \cite{CW, GLY, YYZ}.

\begin{theorem}
\begin{enumerate}[(i)]
\item The space $BMO(\X)$ is the dual space of $H^1(\X)$.
\item The space $BMO_\rho(\X)$ is the dual space of $H^1_\rho(\X)$.
\end{enumerate}
\end{theorem}

\section{The product of functions in $BMO(\mathcal X)$ and $H^1(\mathcal X)$}

Remark that if $g\in \G(\beta,\gamma)$, then 
\begin{equation}\label{bounded property of test functions}
\|g\|_{L^\infty}\leq C \frac{1}{V_1(x_0)}\|g\|_{\G(\beta,\gamma)}
\end{equation}
and 
\begin{equation}\label{integrable property of test functions}
\|g\|_{L^1}\leq (C +\sum_{j=0}^\infty 2^{-j\gamma}) \|g\|_{\G(\beta,\gamma)}\leq C \|g\|_{\G(\beta,\gamma)}.
\end{equation}

\begin{proposition}\label{multipliers for bmo}
Let $\beta\in (0,1]$ and $\gamma\in (0,\infty)$. Then, $g$ is a pointwise multiplier of $BMO(\X)$ for all $g\in \G(\beta,\gamma)$. More precisely, 
$$\|gf\|_{BMO}\leq C \frac{1}{V_1(x_0)}\|g\|_{\G(\beta,\gamma)}\|f\|_{BMO^+}$$
for all $f\in BMO(\X)$. Here and what in follows, 
$$\|f\|_{BMO^+}:= \|f\|_{BMO} + \frac{1}{V_1(x_0)}\int_{B(x_0,1)}|f(x)|d\mu(x).$$
\end{proposition}

Using Proposition \ref{multipliers for bmo}, for $b\in BMO(\X)$ and $f\in H^1(\X)$, one can define the distribution $b\times f\in (\G^{\epsilon}_0(\beta,\gamma))'$ by the rule 
\begin{equation}\label{distribution definition for products}
\langle b\times f, \phi\rangle := \langle \phi b, f\rangle 
\end{equation}
for all $\phi\in \G^{\epsilon}_0(\beta,\gamma)$, where the second bracket stands for the duality bracket between $H^1(\X)$ and its dual $BMO(\X)$.

\begin{proof}[Proof of Proposition \ref{multipliers for bmo}]
By (\ref{bounded property of test functions}) and the pointwise multipliers characterization of $BMO(\X)$ (see \cite[Theorem 1.1]{Na}), it is sufficient to show that
\begin{equation}\label{multipliers for bmo 1}
\log(e +1/r)MO(g, B(a,r))\leq C \frac{1}{V_1(x_0)}\|g\|_{\G(\beta,\gamma)}
\end{equation}
and
\begin{equation}\label{multipliers for bmo 2}
\log(e+ d(x_0,a) + r) MO(g, B(a,r))\leq C \frac{1}{V_1(x_0)}\|g\|_{\G(\beta,\gamma)}
\end{equation}
hold for all balls $B(a,r)\subset \X$. It is easy to see that (\ref{multipliers for bmo 1}) follows from (\ref{bounded property of test functions}) and the Lipschitz property of $g$ (see (ii) of Definition \ref{definition for test functions}). Let us now establish (\ref{multipliers for bmo 2}). If $r<1$, then by (\ref{multipliers for bmo 2}) follows from the Lipschitz property of $g$ and the fact that $\lim_{\lambda\to\infty}\frac{\log(\lambda)}{\lambda^\beta}=0$. Otherwise, we consider the following two cases:
\begin{enumerate}[(a)]
\item The case: $1\leq r\leq \frac{1}{4\kappa^3} d(x_0,a)$. Then, for every $x,y\in B(a,r)$, one has $d(x_0,a)\leq \frac{4\kappa^3}{4\kappa^2-1}$ and $d(x,y)\leq \frac{d(x_0,x)}{2\kappa}$. Hence, the Lipschitz property of $g$ yields
$$|g(x)-g(y)|\leq C \|g\|_{\G(\beta,\gamma)}\frac{1}{V_1(x_0)}\Big(\frac{1}{d(x_0,a)}\Big)^\gamma.$$
 This implies that (\ref{multipliers for bmo 2}) holds since $\lim_{\lambda\to\infty}\frac{\log(\lambda)}{\lambda^\gamma}=0$.
\item The case: $r> \frac{1}{4\kappa^3} d(x_0,a)$. Then, one has $B(x_0,r)\subset B(a, \kappa(4\kappa^3 +1)r)$. Hence, by (\ref{RD-spaces}), we get
\begin{eqnarray*}
\log(e+ d(x_0,a) + r) MO(g, B(a,r)) &\leq& C \frac{\log(2r)}{V_r(x_0)} \|g\|_{L^1}\\
&\leq& C\frac{\log(2r)}{r^{\mathfrak d}}\frac{1}{V_1(x_0)}\|g\|_{\G(\beta,\gamma)}\\
&\leq& C \frac{1}{V_1(x_0)}\|g\|_{\G(\beta,\gamma)}.
\end{eqnarray*}
This proves (\ref{multipliers for bmo 2}) and thus the proof of Propsition \ref{multipliers for bmo} is finished.
\end{enumerate}

\end{proof}

Next we define $L^{\Xi}(\X)$ as the space of $\mu$-measurable functions $f$ such that 
$$\int_{\X} \frac{e^{|f(x)|}-1}{(1+ d(x_0,x))^{2\n}}d\mu(x)<\infty.$$
Then, the norm on the space $L^{\Xi}(\X)$ is defined by
$$\|f\|_{L^{\Xi}}=\inf\left\{ \lambda>0: \int_{\X} \frac{e^{|f(x)|/\lambda}-1}{(1+ d(x_0,x))^{2\n}}d\mu(x)\leq 1\right\}.$$

Recall the following two lemmas due to Feuto \cite{Feu}.

\begin{lemma}\label{Feuto, Lemma 3.2}
For every $f\in BMO(\X)$,
$$\|f - f_{B(x_0,1)}\|_{L^{\Xi}}\leq C \|f\|_{BMO}.$$
\end{lemma}

\begin{lemma}\label{Feuto, Lemma 3.1}
Let $q\in (1,\infty]$. Then,
$$\|(\b - \b_B)\M(\a)\|_{L^1}\leq C \|\b\|_{BMO}$$
for all $\b\in BMO(\X)$ and for all $(H^1,q)$-atom $\a$ related to the ball $B$.
\end{lemma}

The main point in the proof of Theorem \ref{the first main theorem} is the following.

\begin{proposition}\label{key lemma for Orlicz functions}
\begin{enumerate}[(i)]
\item For any $f\in L^1(\X)$ and $g\in L^{\Xi}(\X)$, we have
$$\|fg\|_{L^{\log}}\leq 64{\n}^2 \|f\|_{L^1}\|g\|_{L^{\Xi}}.$$
\item For any $f\in L^1(\X)$ and $g\in BMO(\X)$, we have
$$\|fg\|_{L^{\log}}\leq C \|f\|_{L^1}\|g\|_{BMO^+}.$$
\end{enumerate}
\end{proposition}

\begin{proof}
(i) If $\|f\|_{L^1}=0$ or $\|g\|_{L^{\Xi}}=0$, then there is nothing to prove. Otherwise, we may assume that $\|f\|_{L^1}=\|g\|_{L^{\Xi}}=\frac{1}{8\n}$ since homogeneity of the norms. Then, we need to prove that
$$\int_{\X} \frac{|f(x)g(x)|}{\log(e + |f(x)g(x)|)+\log(e+ d(x_0,x))}d\mu(x)\leq 1.$$
Indeed, by using the following two inequalities
$$\log(e+ ab)\leq 2 (\log(e+a) + \log(e+b)),\; a,b\geq 0,$$
and
$$\frac{ab}{\log(e+ab)}\leq a + (e^b -1),\; a,b\geq 0,$$
we obtain that, for every $x\in \X$,
\begin{eqnarray*}
&&\frac{(1+ d(x_0,x))^{2\n}|f(x)g(x)|}{4\n(\log(e+|f(x)g(x)|) + \log(e+d(x_0,x)))}\\
 &\leq& \frac{(1+ d(x_0,x))^{2\n}|f(x)g(x)|}{2(\log(e+|f(x)g(x)|) + \log(e+(1+d(x_0,x))^{2\n}))}\\
 &\leq& \frac{(1+ d(x_0,x))^{2\n}|f(x)||g(x)|}{\log(e+ (1+ d(x_0,x))^{2\n}|f(x)||g(x)|)}\\
 &\leq& (1+ d(x_0,x))^{2\n}|f(x)| + (e^{|g(x)|} -1).
\end{eqnarray*}
This together with the fact $8\n (e^{|g(x)|}-1)\leq e^{8\n|g(x)|}-1$ give
\begin{eqnarray*}
&&\int_{\X} \frac{|f(x)g(x)|}{\log(e + |f(x)g(x)|)+\log(e+ d(x_0,x))}d\mu(x) \\
&\leq& 4\n \|f\|_{L^1} +\frac{1}{2}\int_{\X} \frac{e^{8\n |g(x)|}-1}{(1+ d(x_0,x))^{2\n}} d\mu(x)\\
&\leq& \frac{1}{2} + \frac{1}{2} =1,
\end{eqnarray*}
which completes the proof of (i).

(ii) It follows directly from (i) and Lemma \ref{Feuto, Lemma 3.2}.

\end{proof}

Now we ready to give the proof for Theorem \ref{the first main theorem}.

\begin{proof}[\bf Proof of Theorem \ref{the first main theorem}]
By (i) of Theorem \ref{atomic decomposition}, $f$ can be written as
$$f=\sum_{j=1}^\infty \lambda_j a_j$$
where the $a_j$'s are $(H^1,\infty)$-atoms related to the balls $B_j$'s and $\sum_{j=1}^\infty |\lambda_j|\leq C \|f\|_{H^1}$. Therefore, for all $b\in BMO(\X)$, we have
\begin{equation}\label{first theorem, integrable function}
\left\|\sum_{j=1}^\infty \lambda_j (b-b_{B_j})a_j\right\|_{L^1} \leq \sum_{j=1}^\infty |\lambda_j| \|(b-b_{B_j})a_j\|_{L^1}\leq C \|b\|_{BMO}\|f\|_{H^1}.
\end{equation}
By this and Definition (\ref{distribution definition for products}), we see that the series $\sum_{j=1}^\infty \lambda_j b_{B_j} a_j$ converges to $b\times f - \sum_{j=1}^\infty \lambda_j (b-b_{B_j})a_j$ in $(\G^{\epsilon}_0(\beta,\gamma))'$. Consequently, if we define the decomposition operators as 
$${\mathscr L}_f (b)= \sum_{j=1}^\infty \lambda_j (b-b_{B_j})a_j$$
and
$${\mathscr H}_f (b)= \sum_{j=1}^\infty \lambda_j b_{B_j} a_j,$$
where the sums are in $(\G^{\epsilon}_0(\beta,\gamma))'$, then it is clear that ${\mathscr L}_f: BMO(\X)\to L^1(\X)$ is a bounded linear operator, since (\ref{first theorem, integrable function}), and for every $b\in BMO(\X)$,
$$b\times f = {\mathscr L}_f(b) + {\mathscr H}_f(b).$$
Now we only need to prove that the distribution ${\mathscr H}_f(b)$ is in $H^{\log}(\X)$. Indeed, by Lemma \ref{Feuto, Lemma 3.1} and (ii) of Proposition \ref{key lemma for Orlicz functions}, we get
\begin{eqnarray*}
\|\M({\mathscr H}_f(b))\|_{L^{\log}} &\leq& \left\| \sum_{j=1}^\infty |\lambda_j| |b_{B_j}| \M(a_j)\right\|_{L^{\log}}\\
&\leq& \left\| \sum_{j=1}^\infty |\lambda_j| |b- b_{B_j}| \M(a_j)\right\|_{L^1} + \left\| b \sum_{j=1}^\infty |\lambda_j|  \M(a_j)\right\|_{L^{\log}}\\
&\leq& C \|f\|_{H^1}\|b\|_{BMO^+}.
\end{eqnarray*}
This proves that ${\mathscr H}_f$ is bounded from $BMO(\X)$ into $H^{\log}(\X)$, and thus ends the proof of Theorem \ref{the first main theorem}.

\end{proof}

\section{The product of functions in $BMO_\rho(\X)$ and $H^1_\rho(\X)$}

For $f\in BMO_\rho(\X)$, a standard argument gives
\begin{equation}\label{relation between BMO spaces}
\|f\|_{BMO^+} \leq C \log(\rho(x_0) +1/\rho(x_0))\|f\|_{BMO_\rho}.
\end{equation}

\begin{proposition}\label{multipliers for generalized bmo associated with the admissible functions}
Let $\beta\in (0,1]$ and $\gamma\in (0,\infty)$. Then, $g$ is a pointwise multiplier of $BMO_\rho(\X)$ for all $g\in \G(\beta,\gamma)$. More precisely, for every $f\in BMO_\rho(\X)$,
$$\|gf\|_{BMO_\rho}\leq C \frac{\log(\rho(x_0)+1/\rho(x_0))}{V_1(x_0)}\|g\|_{\G(\beta,\gamma)}\|f\|_{BMO_\rho}.$$
\end{proposition}

\begin{proof}
By Proposition \ref{multipliers for bmo}, (\ref{relation between BMO spaces}) and (\ref{bounded property of test functions}), we get
\begin{eqnarray*}
\|gf\|_{BMO_\rho} &\leq&  \|gf\|_{BMO} + \|g\|_{L^\infty} \sup_{B\in\D}\frac{1}{\mu(B)}\int_B |f(x)|d\mu(x)\\
&\leq& C \frac{\log(\rho(x_0)+1/\rho(x_0))}{V_1(x_0)}\|g\|_{\G(\beta,\gamma)}\|f\|_{BMO_\rho}.
\end{eqnarray*}

\end{proof}

Using Proposition \ref{multipliers for generalized bmo associated with the admissible functions}, for $b\in BMO_\rho(\X)$ and $f\in H^1_\rho(\X)$, one can define the distribution $b\times f\in (\G^{\epsilon}_0(\beta,\gamma))'$ by the rule 
\begin{equation}\label{distribution definition for products associated with the admissible functions}
\langle b\times f, \phi\rangle := \langle \phi b, f\rangle 
\end{equation}
for all $\phi\in \G^{\epsilon}_0(\beta,\gamma)$, where the second bracket stands for the duality bracket between $H^1_\rho(\X)$ and its dual $BMO_\rho(\X)$.

\begin{proof}[\bf Proof of Theorem \ref{the second main theorem}]
By (ii) of Theorem \ref{atomic decomposition}, there exist a sequence of $(H^1_\rho,\infty)$-atoms $\{a_j\}_{j=1}^\infty$ related to the sequence of balls $\{B(x_j, r_j)\}_{j=1}^\infty$ and  $\sum_{j=1}^\infty |\lambda_j|\leq C \|f\|_{H^1_\rho}$ such that
$$f= \sum_{j=1}^\infty \lambda_j a_j= f_1 + f_2,$$
where $f_1= \sum_{r_j<\rho(x_j)} \lambda_j a_j \in H^1(\X)$ and $f_2= \sum_{r_j \geq \rho(x_j)} \lambda_j a_j$. 

We define the decomposition operators as following
$${\mathscr L}_{\rho,f}(b)= {\mathscr L}_{f_1}(b) + b f_2$$
and
$${\mathscr H}_{\rho,f}(b)= {\mathscr H}_{f_1}(b),$$
where the operators ${\mathscr L}_{f_1}$ and ${\mathscr H}_{f_1}$ are as in Theorem \ref{the first main theorem}. Then, Theorem \ref{the first main theorem} together with (\ref{relation between BMO spaces}) give
\begin{eqnarray*}
\|{\mathscr L}_{\rho,f}(b)\|_{L^1} &\leq& \|{\mathscr L}_{f_1}(b)\|_{L^1} + \sum_{r_j\geq \rho(x_j)} |\lambda_j|\|b a_j\|_{L^1}\\
&\leq& C \|f_1\|_{H^1} \|b\|_{BMO} + C  \|b\|_{BMO_\rho} \sum_{r_j\geq \rho(x_j)} |\lambda_j|\\
&\leq& C \|f\|_{H^1_\rho}\|b\|_{BMO_\rho}
\end{eqnarray*}
and 
$$\|{\mathscr H}_{\rho,f}(b)\|_{H^{\log}}\leq C \|f_1\|_{H^1}\|b\|_{BMO^+}\leq C \|f\|_{H^1_\rho}\|b\|_{BMO_\rho}.$$
This proves that the linear operator ${\mathscr L}_{\rho,f}: BMO_\rho(\X) \to L^1(\X)$ is bounded and the linear operator ${\mathscr H}_{\rho,f}: BMO_\rho(\X) \to H^{\log}(\X)$ is bounded. Moreover,
\begin{eqnarray*}
b\times f &=& b\times f_1 + b\times f_2\\
&=& ({\mathscr L}_{f_1}(b) + {\mathscr H}_{f_1}(b)) + b f_2\\
&=& {\mathscr L}_{\rho,f}(b) + {\mathscr H}_{\rho,f}(b),
\end{eqnarray*}
which ends the proof of Theorem \ref{the second main theorem}.

\end{proof}

\end{document}